\documentclass[11pt,leqno]{amsart}

\usepackage[useregional,showseconds=false,showzone=false]{datetime2}
\usepackage{graphicx}
\usepackage{epstopdf,epsfig}
\usepackage[hidelinks]{hyperref}

\usepackage{xcolor}

\usepackage{palatino}
\usepackage[mathcal]{euler}

\usepackage{dsfont}
\usepackage{amsmath,amsthm,amssymb}
\allowdisplaybreaks

\usepackage[all]{xy}
\xyoption{arc}

\usepackage{amsfonts}
\usepackage{enumerate}

\numberwithin{equation}{subsection}

\swapnumbers

\theoremstyle{plain}
\newtheorem*{theorem*}{Theorem} 
\newtheorem*{lemma*}{Lemma}
\newtheorem*{assumption*}{Assumption}

\newtheorem{theorem}[equation]{Theorem} 
\newtheorem{lemma}[equation]{Lemma}

\theoremstyle{definition}
\newtheorem{definition}[equation]{Definition}

\newtheorem{remark}[equation]{Remark}

\newtheorem*{remark*}{Remark}
\newtheorem*{observation*}{Observation}

\theoremstyle{remark}

\newcommand{\complex}{\mathbb{C}}

\newcommand{\lie}[1]{\mathfrak{#1}}
\newcommand{\compop}{\mathfrak{K}}

\newcommand{\deformation}[2]{\mathbb{N}_{#2}{#1}}

\newcommand{\Compacts}{\mathfrak{K}}
\newcommand{\Bounded}{\mathfrak{B}}
\newcommand{\R}{\mathbb{R}}
\newcommand{\C}{\mathbb{C}}

\DeclareMathOperator{\Ind}{Ind}

\begin{document}

\title[The Mackey Bijection  for   Complex  Reductive  Groups]{The Mackey Bijection  for   Complex  Reductive  Groups   and Continuous Fields of Reduced Group C*-Algebras} 
\author{Nigel Higson}
\author{Angel Rom\'an}
\address{Department of Mathematics, Penn State University, University Park, PA 16802, USA}

% \date{Last processed \DTMnow}

\begin{abstract}
The purpose of this paper is to make a further contribution  to the Mackey bijection for a complex reductive group $G$, between the tempered  dual of $G$ and the unitary dual of the associated Cartan motion group.   We shall construct an embedding of the $C^*$-algebra of the motion group into the reduced $C^*$-algebra of $G$, and  use it  to characterize  the continuous field of reduced group $C^*$-algebras that is associated to the Mackey bijection.  We shall also obtain a new characterization of the Mackey bijection using the same embedding.
\end{abstract}

\maketitle

\section{Introduction}

The \emph{Mackey bijection} is a certain one-to-one correspondence between the tempered  irreducible unitary representations of a real reductive group and the unitary irreducible representations of its Cartan motion group.  Its existence was suggested by George Mackey \cite{Mackey75} in the 1970's (although Mackey stopped short of making a precise conjecture).   After a long pause, Mackey's proposal began to be examined in detail over the past  dozen years, first for complex groups in \cite{Higson08} and ultimately for all real groups in breakthrough work of Afgoustidis \cite{Afgoustidis15}.  

The purpose of this paper is to return to the case of complex groups, and examine the Mackey bijection in greater detail through the theory of group $C^*$-algebras and continuous fields of group $C^*$-algebras.   We shall characterize the continuous field that has played a key role in nearly every study of the Mackey bijection, and use our characterization of the field to give a new characterization of the Mackey bijection.

Group $C^*$-algebras have long played a prominent role in the Mackey bijection.  Indeed
Mackey's original idea was kept alive by Alain Connes, who noticed a link between Mackey's work  and the Connes-Kasparov isomorphism   in $C^*$-algebra $K$-theory \cite{ConnesHigson90,BCH94}.

The connection with $K$-theory is made through the deformation to the normal cone construction from geometry, which associates to the inclusion of a maximal compact subgroup $K$ into any Lie group $G$ a smooth, one-parameter family of Lie groups $G_t$, all of them copies of the group $G$ except at $t{=}0$, where 
\[
G_0 = K \ltimes \bigl ( \operatorname{Lie}(G) / \operatorname{Lie}(K) \bigr )  .
\]
In the case where $G$ is reductive, $G_0$ is the Cartan motion group. Associated to the deformation to the normal cone construction there is a  continuous field of $C^*$-algebras $\{ C^*_r (G_t)  \} _{t \in \R} $.  It turns out that  in the reductive case this continuous field is assembled from \emph{constant} fields of $C^*$-algebras (indeed commutative $C^*$-algebras)  by natural operations: extensions, Morita equivalences and direct limits.  This immediately proves the Connes-Kasparov isomorphism, and by making the assembly process explicit we obtain an explicit Mackey bijection. For details see \cite{Higson08} and \cite{Afgoustidis16}.

To repeat what we wrote earlier, a principal  goal of this paper is to study the continuous field $\{ C^*_r(G_t)\}_{t\in \R}$ in more detail, and indeed to characterize it up to isomorphism in the case where $G$ is a complex reductive group.  
To explain the method, it is helpful to start with a toy model case, in which $G$ is a semidirect product group 
\[
G =K\ltimes V
\]
associated to the action of a compact group $K$ on a real, finite-dimensional vector space $V$ (this is not a reductive group, of course).  Here the deformation to the normal cone associated to the embedding of $K$ into $G$ gives a smooth family of groups $\{ G_t\}$ that is isomorphic to the constant family of groups with fiber $G$. However $\{ G_t\}$ is not \emph{equal} to the constant family; to obtain an isomorphism to the constant family we must use the family of \emph{rescaling morphisms}
\begin{gather*}
\alpha_t \colon G  \longrightarrow G  \\
\alpha_t(k,v) = (k,tv)
\end{gather*}
for $t{\ne}0$.  Applying $\alpha_t$ to the fiber at $t{\ne}0$, and the identity morphism a $t{=}0$, we obtain   an isomorphism from the constant family of groups with fiber $G$ into the deformation to the normal cone family. 
The same rescaling morphisms induce an isomorphism from  the constant field of $C^*$-al\-gebras with fiber $C^*_r(G)$ to the deformation to the normal cone continuous field $\{ C^*_r(G_t)\}$.

Similar rescaling morphisms   do not exist at the group level on a real reductive group, except in trivial cases.   But one might ask whether they nonetheless exist on the reduced group $C^*$-algebra?  

The reason that one might guess that rescaling morphisms exist  at the $C^*$-algebra level is that the structure of the $C^*$-algebra is very closely tied to the structure of the tempered dual of $G$, and there is a natural rescaling operation on the tempered dual. Indeed the tempered dual  is parametrized by a combination of discrete and continuous parameters, with the  latter belonging to vector spaces, or quotients of vectors spaces by finite group actions. So the continuous parameters may be rescaled in the obvious way.  Moreover this rescaling operation plays a central role in the Mackey bijection.

Our first main result is that rescaling morphisms for connected complex reductive groups do indeed exist at the $C^*$-algebra level:

\begin{theorem*}
Let $G$ be a connected complex reductive group.
There is a one-parameter group of automorphisms
\[
\alpha_t \colon C^*_r(G) \longrightarrow C^*_r(G)  \qquad (t>0)
\]
that implements the rescaling action on the tempered dual of $G$ in the Mackey bijection.
\end{theorem*}

See Section~\ref{sec-scaling-automorphism}.  Our second main result is that, at the level of continuous fields, there is a (unique) extension to $t{=}0$, as follows:

\begin{theorem*} Let $G$ be a connected complex reductive group and let $\{ f_t\}$ be a continuous section of the continuous field of $C^*$-algebras associated to the deformation to the normal cone construction for the inclusion of a maximal compact subgroup into $G$.  Then the  limit 
$\lim_{t\to 0} \alpha_t (f_t)$  exists in $C^*_r (G)$, and the formula
\[
\alpha (f_0) = \lim_{t\to 0} \alpha_t(f_t)
\]
defines  an embedding of $C^*$-algebras
\[
\alpha  \colon C^*_r(G_0)\longrightarrow C^*_r(G) .
\]
\end{theorem*}

See Theorems~\ref{thm-the-limit-exists} and \ref{thm-embedding} for the precise statements (we have omitted here   some details related to  the Haar measures on the groups $G_t$, which vary with $t$).

% \begin{remark*}
% Compare  \cite{BravermanKazhdan18} for a very similar phenomenon in the $p$-adic case.
% \end{remark*}

Now, given an embedding of  $C^*$-algebras $\alpha\colon B \to A$, there is a simple and obvious way to construct a continuous field of $C^*$-algebras with fibers 
\[
A_t = \begin{cases} A & t\ne 0 \\ B & t=0,
\end{cases}
\]
namely we take as continuous sections all those set-theoretic sections for which the formula 
\[
t \longmapsto \begin{cases}
a_t & t\ne 0 \\ 
\alpha(b_0) & t=0
\end{cases}
\]
defines a  continuous function from $\R$ to $A$.  Let us call this the \emph{mapping cone} continuous field associated to the inclusion.
Using mapping cones, we are able to characterize the continuous field associated to the deformation to the normal cone, as follows: 

\begin{theorem*}
Let $G$ be a connected complex reductive group.
The continuous field of $C^*$-algebras $\{ \, C^*_r (G_t)\, \}_{t\in \R}$ associated to the deformation to the normal cone construction is isomorphic to the mapping cone field for the embedding 
\[
\alpha  \colon C^*_r(G_0) \longrightarrow C^*_r (G).
\]
 Indeed the morphism 
\[
\{ f_t\} \longmapsto \{ \alpha_t (f_t) \} ,
\]
where $\alpha_0{=}\operatorname{id}$, is a bijection from continuous sections of the deformation to the normal cone field to continuous sections of the mapping cone field.
\end{theorem*}

See Section~\ref{subsec-mapping-cone}  for further details (including the proper treatment of $\alpha_t$ when $t$ is negative).

To summarize, one might say that the continuous field is nothing more or less than the morphism 
\begin{equation*}
% \label{eq-intro-alpha-zero}
\alpha  \colon C^*_r (G_0) \longrightarrow C^*_r (G) .
\end{equation*}
What does this tell us about the Mackey bijection?  Each tempered irreducible unitary representation of $G$ corresponds to an irreducible representation 
\[
\pi \colon C^*_r (G) \longrightarrow \Bounded (H_\pi)
\]
(in fact the range is the ideal of compact operators, $\Compacts (H_\pi)$, but that is not relevant here). The composition of this representation of $C^*_r (G)$ with the embedding $\alpha$ above  is not necessarily irreducible, so composition with  $\alpha$ does not directly determine a map from the tempered dual of $G$ to the unitary dual of $G_0$.  However generically the restriction \emph{is} irreducible, and it turns out that this enough to determine a unique Mackey bijection:

\begin{theorem*}
Let $G$ be a connected complex reductive group.
There is a unique bijection $\mu$ from the tempered unitary dual of $G$ to the unitary dual of $G_0$ with the property that for every $\pi \in G_0$, $\mu(\pi)$ is a subrepresentation of $\pi\circ \alpha$.
\end{theorem*}

See Section~\ref{sec-bijection-characterization}.  It is an outstanding problem to provide some sort of conceptual explanation for the Mackey bijection, and we do not know if the approach in this paper will contribute usefully to the solution.  At the moment, it seems to us that progress in this direction will depend on understanding the rescaling morphisms  $\alpha_t$, or perhaps their infinitesimal generator, more conceptually.  Given the way that the rescaling morphisms are defined (see Section~\ref{sec-scaling-automorphism}) that would, in turn, seem to depend on developing some kind of understanding of the relation between the Mackey bijection and the Plancherel measure, which enters into the formula for the inverse of the Fourier transform isomorphism for $G$.  

As for the problem of extending the approach of this paper to real groups, it will become evident later in this paper that the major issue is to understand the (normalized) intertwining isomorphisms between unitary principal series representations in much greater detail (and thanks to Harish-Chandra's work, this is in turn related to the theory of the Plancherel measure).  In the relatively simple case of complex groups the complexities surrounding intertwiners can be avoided altogether.

\section{Continuous Fields of Reduced Group C*-Algebras}
\label{sec-continuous-fields}

\subsection{Deformation spaces}
\label{subsec-deformation-spaces}

There is a \emph{deformation space} or \emph{deformation to the normal cone}  $\mathbb{N}_VM$ associated to the embedding of any smooth manifold $M$ as a closed submanifold of a smooth manifold $V$.  It is a smooth manifold itself, and it is equipped with a submersion onto the real line.  See for example \cite{Higson08} for an account adapted to the concerns of this paper, although the definition originates elsewhere and much earlier, in algebraic geometry (see \cite{Fulton84} for the algebraic-geometric perspective). 

Let us quickly review the construction in the case of interest to us, which is the inclusion of a maximal compact subgroup $K$ into an almost connected Lie group $G$. As a set, the deformation space is defined to be 
$$
\deformation{K}{G}=N_G K\times \{0\} \, \, \sqcup\,\,  \bigsqcup_{t\ne 0}  G\times \{ t\} 
$$
where $N_G  K$ is the normal bundle to $K$, 
\[
N_G K = TG\vert _K \,\big / \,TK.
\]
The deformation space  has a unique smooth manifold structure for which
\begin{enumerate}[\rm (i)]

\item The natural map $N_G K\to \R$ is smooth.

\item If $f$ is a smooth function  on $G$, then the function 
\[
\begin{cases} (g,t) \mapsto f(g) & t \ne 0 \\
(X_k,0)\mapsto f(k)
\end{cases}
\]
is smooth on $\mathbb{N}_G K$.

\item If $f$ is a smooth function  on $G$,  and if $f$ vanishes on $K$, then the function 
\[
\begin{cases} (g,t) \mapsto t^{-1} f(g) & t \ne 0 \\
(X_k,0)\mapsto X_k(f)
\end{cases}
\]
is smooth on $\mathbb{N}_G K$.

\item At every point, local coordinates can be selected from functions of the above types.

\end{enumerate} 
The commuting diagram of multiplication operations 
\[
\xymatrix{ K\times K \ar[r] \ar[d]& K\ar[d]  \\ 
G \times G \ar[r]&G 
}
\]
induces a diagram 
\[
\xymatrix{ TK\times TK \ar[r] \ar[d]& TK\ar[d]  \\ 
TG \times TG \ar[r]&TG  ,
}
\]
and we obtain from this a multiplication operation 
\[
N_G K \times N_G K \longrightarrow N_G K.
\]
It makes the normal bundle into a Lie group, and we obtain from $\mathbb{N}_GK$ a smooth family of Lie groups over $\R$ (that is, a Lie groupoid with object space $\R$ and source and target maps equal to one another).

If we  trivialize the tangent bundles on $G$ and $K$ by left translations, then we obtain an identification of bundles and Lie groups 
\[
N_G K \cong K \ltimes \lie{g}/ \lie {k}
\]
where on the right is the semidirect product group associated to the adjoint action.

\subsection{The associated continuous field of C*-algebras}

For $t{\in}\R$ let  us denote by $G_t $ the fiber over $t$ of the submersion 
\[
\mathbb{N}_G K \longrightarrow \R.
\]
It is a Lie group, and for  $t{\ne} 0$ it is  isomorphic to $G$ itself. 
Choose a Haar measure $\mu$  for $G$. Then equip the fibers    $G_t$ for $t{\neq} 0$ with the left Haar measure   $dg_t = |t|^{-d}dg$, 
where 
\begin{equation}
\label{eq-def-of-d}
d = \dim (G/K)
\end{equation}
(we shall use this notation throughout the paper).  The Haar measure on $G$ determines   an associated Haar measure $dg_0$ on $G_0$, and altogether the Haar measures that we have chosen for the fibers $G_t$ vary smoothly with $t$, in the sense that if $f$ is a smooth and compactly supported function on $\mathbb{N}_G K$, and if $f_t$ denotes the restriction of $f$ to $G_t$, then the integral 
\[
\int_{G_t} f_t(g_t) \; dg_t  
\]
is a smooth function of $t{\in} \R$.   Moreover if $f_1$ and $f_2$ are two smooth and compactly supported functions on $\mathbb{N}_G K$, then the convolution product 
\[
(f_{1,t}*f_{2,t}) (g_t) =  \int _{G_t } f_{1,t} (\gamma_t ) f_{2,t}(\gamma_t ^{-1}g_t)\; d  \gamma_t \qquad (g_t\in G_t)
\]
is a smooth and compactly supported function $\mathbb{N}_G K$, too.

The same convolution formula applied to functions on $G_t$ alone defines a product on $C_c^\infty(G_t)$, and   we shall denote by $C^*_r(G_t)$ the reduced $C^*$-algebra completion (in the norm that $C_c^\infty (G_t)$  obtains as left-convolution operators on $L^2 (G_t,dg_t)$, and the adjoint operation it obtains from the operator adjoint operation).    Compare  \cite[Sec.\;7.2]{Pedersen79}.

\begin{lemma}[See {\cite[Lemma 6.13]{Higson08}}] 
Let $G$ be an almost-connected Lie group and let $K$ be a maximal compact subgroup of $G$.
If $f$ is a smooth and compactly supported function on the deformation space $\deformation{G}{K}$, and if $f_t$ denotes its restriction to $G_t$, then the norm
 $\|f_t\|_{C^*_r (G_t)}$ 
is a continuous function of $t\in \R$. \qed
\end{lemma}

It follows that the smooth and compactly supported functions on $\mathbb{N}_G K$ generate the continuous sections of a continuous field of $C^*$-algebras over $\R$ with fibers $C^*_r (G_t)$ in the sense of \cite[Prop.\;10.2.3]{Dixmier77}.  We shall call this the \emph{deformation} continuous field of reduced group $C^*$-algebras associated to the inclusion of $K$ into $G$.

It will be helpful to think of the  $C^*$-algebras $C^*_r(G_t)$ as being  completely distinct from one another, even though for $t{\ne}0$ they may be viewed as completions of the same space of functions in equivalent norms.  When we wish to compare  $C^*_r(G_t)$ for different $t{\ne}0$, as we shall in Section~\ref{sec-limit-formula-and-embedding}, we shall do so using the canonical $*$-isomorphisms
\begin{equation}
\label{eq-lambda-t-isomorphism}
\lambda_t \colon C^*_r (G_t)\stackrel \cong \longrightarrow C^*_r(G)
\end{equation}
associated to the isomorphisms $G_t{\cong }G$. These are given by the formulas 
\[
\lambda_t \colon f \longmapsto \Bigl [ g \mapsto |t|^{-d} f(g,t)
\Bigr ]
\]
for $f\in C_c^\infty (G_t)$.  The factor $|t|^{-d}$ accounts for the change in Haar measures.   

The isomorphisms $\lambda_t$ may also be described as follows. For $t{\ne}0$ the left translation  action of $G_t{\cong} G$ on itself integrates to a $C^*$-algebra representation
\[
\lambda_t \colon C^*_r (G_t)  \longrightarrow \mathfrak{B}(L^2 (G,dg)) .
\]
The image is independent of $t$, and is the image of the regular representation of $C^*_r(G)$ itself.  So we obtain $*$-isomorphisms as above.

\subsection{The case of reductive groups}
\label{subsec-deformations-for-reductive-groups}
If $G$ is a real reductive group, then the deformation space $\mathbb{N}_GK$ may be given a more concrete form using the following basic structural facts  about $G$.
Fix  a \emph{Cartan decomposition}  $\lie{g} = \lie{k}\oplus \lie{s}$ for the Lie algebra of $G$.  Then of course $\lie{g}/\lie{k}\cong \lie{s}$ and 
\begin{equation}
\label{eq-motion-group-and-cartan-decomp}
K\ltimes \lie{g}/\lie{k}\cong K\ltimes \lie{s}.
\end{equation}
Next, choose  a maximal abelian subspace of $\lie{a}\subseteq \lie{s}$ and an Iwasawa decomposition 
\[
G = K A N
\]
where $A = \exp [ \lie{a}]$ (any two Iwasawa decompositions are conjugate by an element of $K$).  The smooth map 
\begin{gather*}
K \times \lie{a} \times \lie{n} \longrightarrow G \\
(k,X,Y)\longmapsto k \exp(X)\exp(Y)
\end{gather*}
is then a diffeomorphism.  See for example \cite{Knapp02} for all this. 

If we view $K$ as a submanifold of $K{\times}\lie{a}{\times}\lie{n}$ via the inclusion
\begin{gather*}
K \longrightarrow K \times \lie{a} \times \lie{n}  \\
k \longmapsto (k , 0, 0) ,
\end{gather*}
then the functoriality of the deformation space construction (see for example \cite[p.303]{HSSHigson18}) and the commutativity of the diagram 
\[
\xymatrix{
K \ar@{=}[r]\ar[d] & K\ar[d]\\
 K \times \lie{a} \times \lie{n} \ar[r]& G
}
\]
allows us to view $\mathbb{N}_GK$ as the deformation space associated to the inclusion of the zero section into a (trivial) vector bundle.  For the latter, see \cite[Ex.\;4.6]{Higson10}. We obtain a diffeomorphism
\[
K \times \lie{a}\times \lie{n} \times \R  \longrightarrow \mathbb{N}_GK
\]
given by the formula 
\[
(k,X,Y,t) \longrightarrow 
	\begin{cases} 
	\bigl ( k \exp(tX)\exp(tY), t \bigr ) & t \ne 0 \\
	\bigl((k,[X{+}Y]),0\bigr ) & t = 0.
	\end{cases}
\]
Here  $[X{+}Y]$ denotes the class in the quotient vector space $\lie{g}/\lie{k}$ associated to the vector $X{+}Y \in \lie{g}$, or equivalently, keeping in mind the isomorphism \eqref{eq-motion-group-and-cartan-decomp}, the projection of $X{+}Y\in \lie{a}\oplus\lie{n} $ onto the summand $\lie{s}$ in the Cartan decomposition $\lie{g}=\lie{k}\oplus\lie{s}$.
 
This computation has the following consequence that will be of central importance later on:
 
\begin{lemma} 
\label{lem-smooth-function-on-def-space}
A function $f\colon \mathbb{N}_GK \to \C$ is smooth and compactly supported if and only if the function  
\[
(k,X,Y,t) \longmapsto 
	\begin{cases}
	f(k\exp(t^{-1}X)\exp(t^{-1}Y)) & t \ne 0 \\
	f(k, [X{+}Y]),0)  &  t = 0
	\end{cases}
\]
is smooth and compactly supported on $K{\times}\lie{a}{\times} \lie{n}{\times} \R$.\qed
\end{lemma}

\section{Structure of Reduced Group C*-Algebras}
\label{sec-group-algebras}

The main purpose of this section is to review  the detailed description of the $C^*$-algebra $C^*_r (G)$, which was presented first in \cite{PeningtonPlymen83}.   This we shall use throughout the paper.  We shall also review the classification of the irreducible unitary representations of the motion group, due to Mackey \cite{Mackey49}, which we shall use in the final section.  For completeness we shall also review the structure of $C^*_r (G_0)$; this will not play a significant role in what follows, but among other things the unitary dual of $G_0$ is most easily determined from the structure of $C^*_r (G_0)$.

\subsection{Connected complex reductive groups}

Let $G$ be a complex, connected, reductive group.  To describe its tempered unitary dual and reduced group $C^*$-algebra we begin with an Iwasawa decomposition $G=KAN$. Let $M$ be the centralizer in $K$ of $A$, which is a maximal torus in $K$, and let $P$ be the (\emph{Borel} or \emph{minimal parabolic}) subgroup $MAN$.  The group $N$ is a normal subgroup of $P$, and as a result of this, if $\sigma$ is a character of $M$, and if $\nu\in \mathfrak{a}^*$, then the formula 
\[
\sigma\otimes \exp(i\nu) \colon m\cdot \exp(X)\cdot n \longmapsto \sigma(m)\cdot \exp(i\nu(X))
\]
defines a unitary character of $P$.  The \emph{unitary principal series representation} of $G$ associated to $\sigma$ and $\nu$ (and the given choice of minimal parabolic subgroup) is the unitarily induced representation 
\[
\pi_{\sigma,\nu} = \Ind_P^G \;\sigma\otimes \exp(i \nu).
\]  
We shall need the following explicit description of this representation.   The Hilbert space on which it acts is the completion of the space of all smooth functions $\psi \colon G \to \C$ such that 
\begin{equation}
\label{eq-function-in-induced-rep-space}
 \psi(gm\exp(X)n)=e^{-(\rho+i \nu)(X)}\sigma(m)^{-1}\psi(g)
 \end{equation}
for all $g{\in} G$, $m{\in} M$, $X{ \in} \lie{a}$ and $n{\in} N$,  in the norm  associated to the inner product
\[
 \langle \phi,\psi \rangle = \int_K \overline{\phi(k)} \psi(k)\; d k .
 \]
Here $
\rho\in \lie{a}^*$ is the half-sum of the positive restricted roots  associated to our choice of $N$ (its presence makes the representation $\pi_{\sigma,\nu}$ unitary; other than that we shall not need any further information about it).   The action of $G$ on the functions \eqref{eq-function-in-induced-rep-space} is by left translation. 

Thanks to the Iwasawa decomposition, a function satisfying  \eqref{eq-function-in-induced-rep-space} is completely determined by its restriction to $K$, and we find that the Hilbert space of our unitary principal series representation identifies, via restriction of functions to $K$, with the Hilbert space 
\[
L^2(K)^{\sigma}= \bigl \{ \, \psi :K \stackrel{L^2}\rightarrow \complex : \psi (km)= \sigma(m)^{-1}\psi (k)\;\quad  \forall m\in M, \; \forall k \in K  \, \bigr \} .
\]
We shall work exclusively with this realization from now on.  An advantage of this realization  is that the Hilbert space for the principal series representation associated to $\sigma$ and $\nu$ depends only on $\sigma$ and not on $\nu$.  A disadvantage is that the formula for the action of $G$ is a bit more complicated.  To describe it we need the Iwasawa decomposition of elements of $G$: 
\begin{equation}
\label{eq-iwasawa-for-elements}
\qquad   g= \kappa(g)\cdot e^{H(g)}\cdot n \quad  ( \kappa(g) \in K, \quad X\in \lie{a},\quad n \in N).
\end{equation}
With this notation,
\begin{equation}
\label{eq-action-of-g-in-compact-form}
( \pi_{\sigma,\nu}(g)\phi) (k)=e^{-(\rho+i \nu)H(g^{-1}k)}  \phi(\kappa(g^{-1}k))
\end{equation}
for $\phi\in L^2(K)^\sigma$.

The Weyl group of $K$,
\begin{equation}
\label{eq-weyl-group}
W = N_K(M) / M ,
\end{equation}
acts by conjugation on  $M$, of course, but it also acts by conjugation on $A$ (this is because the complexification of $M$ is $MA$).  So $W$ acts on the parameter space $\widehat M \times \lie{a}^*$ for the unitary principal series.  The main facts  about the unitary principal series of $G$ may then be summarized as follows:

\begin{theorem}
Let $G$ be a connected, complex reductive group.
\begin{enumerate}[\rm (i)]

\item 
Every unitary principal series representation is an irreducible, tempered unitary representation of $G$

\item Every irreducible, tempered unitary representation of $G$ is equivalent to some unitary principal series representation.

\item  Two  unitary principal series representations $\pi_{\sigma',\nu'}$ and $\pi_{\sigma'',\nu''}$ are equivalent if and only if  
the pairs $(\sigma',\nu'),(\sigma'',\nu'')\in \widehat M {\times} \lie{a}^*$ are conjugate by an element of $W$. \qed
\end{enumerate}

\end{theorem}

These facts are essentially due to Gelfand and Naimark when $G$ is a  classical group \cite{GelfandNaimark50}, and to Harish-Chandra \cite{Harish-Chandra54} in general, although the irreducibility of the unitary principal series in full generality is due to Wallach
\cite{Wallach71}. (The question of whether or not  the unitary principal series defines a \emph{closed} subset of the unitary dual, and hence accounts for  the full tempered dual, was not initially addressed.  But see for example \cite{CCH16} for this issue.)

\begin{definition}
\label{def-iota-notation}
Let  $\sigma\in \widehat M$ and let $\nu \in \mathfrak{a}^*$.   We shall write 
\[
\iota (\sigma,\nu) \in \widehat G_r
\]
for the equivalence class of the irreducible unitary principal series representation
\[
\pi_{\sigma,\nu} = \Ind_{P}^{G} \sigma \otimes e^{i \nu}
\]
of the   group $G$.
\end{definition}

Now for  $f\in C_c^\infty (G)$, we define
$$
\pi_{\sigma,\nu}(f)\phi=\int_G f(g) \pi_{\sigma,\nu}(g)\phi\; dg
$$
for $\phi\in L^2(K)^\sigma$.   The formula defines a $C^*$-algebra representation 
\[
\pi_{\sigma,\nu} \colon C^*_r (G) \longrightarrow \mathfrak{B} (L^2 (K)^\sigma ).
\]

\begin{theorem}[See for example {\cite[Cor.\;4.12]{CCH16}}]
\label{thm-reductive-riemann-lebesgue}
Let $G$ be  a connected complex reductive group and let $\sigma\in \widehat M$.   There is a  $C^*$-algebra morphism
\[
  \pi_\sigma \colon C_r^*(G)  \longrightarrow   C_0  (\mathfrak{a}^*,\compop(L^2(K)^\sigma)  )
\]
such that $\pi_\sigma(f)(\nu) = \pi_{\sigma,\nu}(f)$ for every $\nu\in \lie{a}^*$. \qed
\end{theorem}

Now let  $\sigma\in \widehat M$, and denote by $W_\sigma$ the isotropy group of $\sigma$ in the Weyl group $W= N_K(M)/M$.  Associate to $\sigma$ its infinitesimal form  
\[
\sigma\colon \lie{m} \longrightarrow  i \R ,
\]
 and then extend this to a linear functional 
\begin{equation}
\label{eq-extended-sigma}
\sigma\colon \lie{k} \longrightarrow i \R
\end{equation}
by writing $\lie{k} = \lie{m} \oplus \lie{m}^ \perp$ using any $K$-invariant inner product, and defining $\sigma$ to be zero on $\lie{m}^\perp$.  
We may then identify $W_\sigma$ with the Weyl group of the isotropy group $K_\sigma$ for \eqref{eq-extended-sigma}. It follows, in  particular, that $W_\sigma$ is itself a Weyl group. Moreover since $G$ is a complex group, we have that 
\[
\lie{s} = i\cdot \lie{k} \quad \text{and}\quad 
\lie{a} = i \cdot \lie{m}
\] 
So we can divide $\lie{a}^*$ into Weyl chambers for the action of $W_\sigma$, and choose one \emph{positive Weyl chamber} 
\begin{equation}
\label{eq-positive-sigma-weyl-chamber}
\lie{a}^*_{+,\sigma}\subseteq \lie{a}^*
\end{equation}
which is a fundamental domain for the action of $W_\sigma$ on $\lie{a}^*$.

The  morphisms $\pi_\sigma$ in Theorem~\ref{thm-reductive-riemann-lebesgue} determine morphisms
\[
  \pi_\sigma \colon C_r^*(G)  \longrightarrow   C_0  (\mathfrak{a}^*_{+,\sigma},\compop(L^2(K)^\sigma)  )
\]
 by restriction to the positive Weyl chamber. These assemble into a $C^*$-algebra isomorphism, as follows:

\begin{theorem}
\label{thm-structure-reductive-reduced-algebra}
Let $G$ be  a complex reductive group.  The representations of $G$ in the unitary principal series induce a $C^*$-algebra isomorphism
\[
\oplus_{\sigma\in\widehat{M}_{+}} \pi_\sigma \colon C_r^*(G)\stackrel \cong \longrightarrow  \bigoplus_{\sigma\in \widehat M_{+}} C_0\bigl (\mathfrak{a}^*_{+,\sigma},\compop(L^2(K)^\sigma)\bigr ) ,
\]
where $\widehat{M}_+ \subseteq \widehat{M}$ is a positive Weyl chamber for the action of $W$.
\end{theorem}

\begin{proof} 
The general considerations in \cite{CCH16} show that there is an isomorphism
\[
\oplus_{\sigma\in\widehat{M}_{+}} \pi_\sigma \colon C_r^*(G)\stackrel \cong \longrightarrow  \bigoplus_{\sigma\in \widehat M_{+}} C_0\bigl (\mathfrak{a}^*,\compop(L^2(K)^\sigma)\bigr )^{W_\sigma},
\]
where $W_\sigma$ acts by intertwining automorphisms on $C_0 (\mathfrak{a}^*,\compop(L^2(K)^\sigma) )$. The required result now follows from the fact that $\lie{a}^*_{+,\sigma}$ is a fundamental domain for the action of $W_\sigma$ on $\lie{a}^*$.
\end{proof}

\subsection{The motion group}
The Cartan motion group 
\[
G_0 = K\ltimes \lie{g}/\lie{k} \cong K\ltimes \lie{s}
\]
 is amenable, and so its tempered dual is equal to the full unitary  dual, and the  canonical morphism $C^*(G_0)\to C^*_r (G_0)$ from the full to the reduced group $C^*$-algebra is an isomorphism.  We shall use $C^*_r(G_0)$ rather than $C^*(G_0)$ in what follows simply to be consistent with our usage for the complex reductive group $G$, which is not amenable unless it is abelian.

Let $\nu \in \lie{s} ^*$.  The function $\exp(i \nu)$ is  a unitary character on the additive group $\lie{s}$, and so we may form the unitarily induced representation 
\[
\pi_\nu = \Ind _{\lie{s}}^{K\ltimes \lie{s}} exp(i \nu)
\]
 of  the motion group $G_0$.  By definition, its  Hilbert space is the completion of the space of smooth functions $\psi\colon G_0\to \C$ such that 
\[
\psi(k,X) = \psi(k)\exp(-i \nu(X)) \qquad (\forall (k,X)\in G_0)
\]
in the norm induced from the inner product 
\[
\langle \phi, \psi\rangle = \int _K \overline{\phi(k)} \psi(k)\; dk .
\]
The action of $G_0$ is by left translation.

Of  course the Hilbert space  identifies with $L^2 (K)$.  Under this identification the subgroup $K\subseteq G_0$ acts by left translation, whereas an element $X\in\lie{s}$ acts by pointwise multiplication  by the function $k\mapsto \exp(i \nu(k^{-1}\cdot X))$.

 The unitary representation $\pi_\nu$ of $G_0$ integrates to a $C^*$-algebra representation 
\[
\pi_\nu \colon C^*_r(G_0)\longrightarrow \Compacts ( L^2 (K)) ,
\]
and the Riemann-Lebesgue lemma for the ordinary  Fourier transform implies the following result: 

\begin{lemma}
\label{lem-riemann-lebesgue1}
There is a morphism $C^*$-algebras
\[
\pi \colon C^*_r(G_0) \longrightarrow C_0(\lie{s}^*,\compop(L^2(K)))
\]
such that 
\[
\pushQED{\qed} 
\pi(f)(\nu) = \pi_\nu (f) \qquad (\forall \nu \in \lie{s}^*).
\qedhere
\popQED
\]
\end{lemma}
  
Now the  group $K$ acts on $\lie{s}^*$ by the coadjoint representation and on $L^2 (K)$ by right translation, and these actions combine to give an action of $K$ on the $C^*$-algebra $C_0(\lie{s}^*,\compop(L^2(K)))$ by $C^*$-algebra automorphisms.   The  following result describes the $C^*$-algebra $C^*_r(G_0)$ up to isomorphism:

\begin{theorem}
[See for example {\cite[Thm\,3.2]{Higson08}}]
The morphism $\pi$ in Lemma \textup{\ref{lem-riemann-lebesgue1}}  induces an isomorphism
\[
\pushQED{\qed} 
C^*_r(G_0)\stackrel \cong \longrightarrow C_0(\lie{s}^*,\compop(L^2(K)))^K.
\qedhere
\popQED
\]
\end{theorem}

The representations $\pi_\nu$ are not irreducible, but using the theorem it is not difficult to obtain the following description of the  irreducible unitary representations of $G_0$, which is a simple case of a general result due to Mackey (see \cite{Mackey49} or \cite{Mackey55}).    See also \cite[Sec.\,3.1]{Higson08}.

\begin{theorem}
\label{thm-mackey-classification}
If  $\nu\in \lie{s} ^*$ and if $\tau\in \widehat K_\nu$, then the  representation 
\[
 \Ind_{K_{\nu}\ltimes \lie{s}}^{K \ltimes \lie{s}} \tau \otimes e^{i \nu}
 \]
 is irreducible.   Every irreducible unitary representation of $G_0$ is equivalent to one that is obtained in this way, and two  representations  
 \[
  \Ind_{K_{\nu'}\ltimes \lie{s}}^{K \ltimes \lie{s}} \tau' \otimes e^{i \nu'}
  \quad \text{and} \quad 
  \Ind_{K_{\nu''}\ltimes \lie{s}}^{K \ltimes \lie{s}} \tau'' \otimes e^{i \nu''}
  \]
   are unitarily equivalent if and only if the data $(\tau',\nu')$ and $(\tau'',\nu'')$ are conjugate by an element of $K$.  \qed
\end{theorem}
 
Now let $\lie{a}\subseteq \lie{s}$ and $M\subseteq K$ be as in the previous subsection.   Because $G$ is a complex group, we can form $ i\cdot \lie{k} \subseteq \lie{g}$, and in fact 
\[
\lie{s} = i \cdot \lie{k} ,
\]
from which it follows that $\lie{a} =i \cdot \lie {m}$, where $\lie{m}$ is the Lie algebra of $M$.   It follows from standard facts about compact groups that every element of $\lie{s}$ is conjugate by an element of $K$ to an element of $\lie{a}$, and that two elements in $\lie{a}$ are conjugate to one another by an element of $K$ if and only if they are conjugate by an element of the Weyl group 
$
W = N_K(M) / M $.

We can write 
\[
\lie{s} = \lie{a} \oplus \lie{a}^\perp
\]
using a $K$-invariant inner product on $\lie{s}$ (the orthogonal complement does not depend of the choice of inner product), and so regard $\lie{a}^*$ as a subspace of $\lie{s}^*$.  We see from Theorem~\ref{thm-mackey-classification}, therefore, that every irreducible representation of $G_0$ is equivalent to one of the form 
\[
  \Ind_{K_{\nu}\ltimes \lie{s}}^{K \ltimes \lie{s}} \tau \otimes e^{i \nu}
\]
for some $\nu\in \lie{a}^*$ and some $\tau \in \widehat K_\nu$

Finally, $K_\nu$ is a connected compact Lie group, and $M\subseteq K_\nu$ is a maximal torus.  So the irreducible representations of $K_\nu$ are parametrized by their highest weights, which are orbits in $\widehat M$ of the Weyl group 
\[
W_\nu = N_{K_\nu}(M)/M,
\]
which is also the isotropy group of $\nu\in \lie{a}^*$ for the action of $W$ on $\lie{a}^*$.  So the irreducible unitary representations of $G_0$ are parametrized by elements of the set 
\[
\Bigl (\,  \bigsqcup_{\nu \in \lie {a}^*} \widehat{M}/ W_\nu \times \{\nu\} \, \Bigr ) / W = \bigl (\,  \widehat {M} \times \lie{a}^* \, \bigr ) / W .
\]
For future use we shall introduce the following notation for irreducible representation attached to the parameter $(\sigma, \nu)$:
 
\begin{definition}
Let  $\sigma\in \widehat M$ and let $\nu \in \mathfrak{a}^*$. Let  $\tau_\sigma\in \widehat K_{\nu}$ be the irreducible representation with highest weight $[\sigma]\in \widehat{M}/ W_\nu$.  We shall write 
\[
\pi(\sigma,\nu) \in \widehat G_0
\]
for the equivalence class of the irreducible unitary representation
\[
\Ind_{K_{\nu}\ltimes \lie{s}}^{K \ltimes \lie{s}} \tau_\sigma \otimes e^{i \nu}
\]
of the motion group $G_0$.
\end{definition}

\section{Scaling automorphisms}
\label{sec-scaling-automorphism}

%\subsection{Definition of $\alpha_t$}
In this section we shall construct a one-parameter group of automorphisms
\[
\alpha_t \colon C^*_r (G) \longrightarrow C^*_r (G) .
\]
The automorphisms will be  parametrized by the multiplicative group of positive real numbers, rather than the usual additive group of real numbers, and so the group law is $\alpha_{t_1} \circ \alpha_{t_2} = \alpha _{t_1t_2}$.

\subsection{Definition of the scaling automorphisms}
Given the structure theory for $C^*_r (G)$ that was presented in the previous section, the  construction is extremely simple. 
Let $\sigma \in \widehat{M}_+$. Define, for $t>  0$, an automorphism
\[
\alpha_{\sigma,t}:C_0 \bigl ( \lie{a}^*_{+,\sigma},\compop(L^2(K)^\sigma) \bigr ) 
\longrightarrow 
C_0 \bigl ( \lie{a}^*_{+,\sigma},\compop(L^2(K)^\sigma) \bigr )
\]
by 
\begin{equation}
\label{eq-rescaling-automorphism-def1}
\alpha_{\sigma,t}(f)(\nu)=f(t^{-1}\nu).
\end{equation}
The individual one-parameter groups $\alpha_{\sigma,t}$ may be combined by direct sum into a one-parameter group of automorphisms  
\[
  \oplus _\sigma \alpha_{\sigma,t} \colon 
\bigoplus_{\sigma\in \widehat M_+} C_0 \bigl ( \lie{a}^*_{+,\sigma},\compop(L^2(K)^\sigma) \bigr ) 
\longrightarrow 
\bigoplus_{\sigma\in \widehat M_+} C_0 \bigl ( \lie{a}^*_{+,\sigma},\compop(L^2(K)^\sigma) \bigr ) ,
\]
and then we define   automorphisms $\alpha_t$ of $C^*_r (G)$ by means of the commuting diagram 
\begin{equation}
\label{eq-rescaling-automorphism-def2}
\xymatrix@C=50pt{
C^*_r(G) \ar[d]_{\oplus_\sigma \pi_\sigma}^\cong  \ar[r]^{\alpha_t} & C^*_r(G) \ar[d]^{\oplus_\sigma \pi_\sigma}_\cong 
\\
\bigoplus_{\sigma\in \widehat M_+} C_0 \bigl ( \lie{a}^*_{+,\sigma},\compop(L^2(K)^\sigma) \bigr )
\ar[r]_{\oplus_\sigma \alpha_{\sigma,t} }& 
\bigoplus_{\sigma\in \widehat M_+} C_0 \bigl ( \lie{a}^*_{+,\sigma},\compop(L^2(K)^\sigma) \bigr ) .
}
\end{equation}

\subsection{Scaling automorphisms for negative t}

As we shall soon see, the key property of the rescaling automorphism $\alpha_t$, which is immediate from its definition, is that if $\sigma{\in}\widehat{M}_{+}$ and if $\nu{ \in} \lie{a}^*_{\sigma,+}$, then 
\begin{equation}
\label{eq-key-property-of-alpha-t}
\pi_{\sigma,\nu}(\alpha_t(f)) = \pi_{\sigma, t^{-1}\nu}(f)
\end{equation}
for all $f{\in} C^*_r (G)$.
We shall want to extend this to negative $t$, and to this end we define automorphisms 
\[
\alpha_t \colon C^*_r (G) \longrightarrow C^*_r (G)  
\]
for $t{<}0$ as follows.  

First we define $\lie{a}^*_{-,\sigma}$ to be the negative of the Weyl chamber $\lie{a}^*_{+,\sigma}$.  This is simply another Weyl chamber for $W_\sigma$, and so all the constructions that we made in Section~\ref{sec-group-algebras} using $\lie{a}^*_{+,\sigma}$ can be repeated for $\lie{a}^*_{-,\sigma}$.  In particular there is an isomorphism of $C^*$-algebras 
\[
\oplus_{\sigma\in\widehat{M}_{+}} \pi_\sigma \colon C_r^*(G)\stackrel \cong \longrightarrow  \bigoplus_{\sigma\in \widehat M_{+}} C_0\bigl (\mathfrak{a}^*_{-,\sigma},\compop(L^2(K)^\sigma)\bigr ) ,
\]
We now define $\alpha_t\colon C^*_r(G)\to C^*_r (G)$ for $t{<}0$ by means of the commuting diagram
\begin{equation*}
\xymatrix@C=50pt{
C^*_r(G) \ar[d]_{\oplus_\sigma \pi_\sigma}^\cong  \ar[r]^{\alpha_t} & C^*_r(G) \ar[d]^{\oplus_\sigma \pi_\sigma}_\cong 
\\
\bigoplus_{\sigma\in \widehat M_+} C_0 \bigl ( \lie{a}^*_{-,\sigma},\compop(L^2(K)^\sigma) \bigr )
\ar[r]_{\oplus_\sigma \alpha_{\sigma,t} }& 
\bigoplus_{\sigma\in \widehat M_+} C_0 \bigl ( \lie{a}^*_{+,\sigma},\compop(L^2(K)^\sigma) \bigr ) .
}
\end{equation*}
where $\alpha_{\sigma,t}(h)(\nu) = h(t^{-1}\nu)$.  The key property \eqref{eq-rescaling-automorphism-def1} now holds for all $t{\ne}0$, for all $\nu\in \lie{a}^*_{+,\sigma}$ and all $f\in C^*_r (G)$.

\section{Limit Formula and Embedding}
\label{sec-limit-formula-and-embedding}

\subsection{Limit formula}

The main result of this section links the scaling automorphisms 
\[
\alpha_t \colon C^*_r (G)\to C^*_r (G)
\]
from \eqref{eq-rescaling-automorphism-def2} with the regular representations 
\[
\lambda _t \colon C^*_r (G_t) \to C^*_r(G)
\]
from \eqref{eq-lambda-t-isomorphism}  as follows:

\begin{theorem}
\label{thm-the-limit-exists}
If  $\{ f_t\} $ is any continuous section of the continuous field  $\{C_r^*(G_t\}$, then the limit 
\[
\lim_{t\to 0} \alpha_t ( \lambda_t ( f_t))
\]
exists in $C^*_r (G)$.
\end{theorem}

\begin{remark} Of course we exclude the value $t{=}0$ in forming the limit.
\end{remark}

  We shall proving the theorem by carrying out an explicit computation with a suitable collection of continuous sections.
  To this end, recall that a collection $\mathcal F$ of continuous sections of the continuous field of $C^*$-algebras $\{ C^*_r (G_t)\}$ is called a \emph{generating family} if for every continuous section $s$, every $\varepsilon > 0$ and every $t_0\in \R$ there is some element $f \in \mathcal{F}$ and a neighborhood $U$ of $t_0\in \R$ such that 
  \[
  t \in U \quad \Rightarrow \quad \| f(t) - s(t) \| < \varepsilon.
  \]
  The following is an immediate consequence of the fact that $C^*$-algebra isomorphisms are isometric:
  \begin{lemma} 
  \label{lem-reduce-to-generating-family}
  If the limit in Theorem~\textup{\ref{thm-the-limit-exists}} exists for a generating family of continuous sections of $\{C^*_r (G_t)\}$, then it exists for all continuous sections of $\{ C^*_r (G_t)\}$. \qed
  \end{lemma}

Recall that if $K$ acts   continuously on a complex vector   space $W$, then a vector   $w{\in}W$ is said to be \emph{$K$-finite} if the linear span of  the orbit of $w$ under the action of $K$ is finite-dimensional and the action on this finite-dimensional space is continuous, or equivalently if $ w$   lies in the image under the natural map 
\[
\bigoplus_{\tau \in \widehat K}V_\tau \otimes _{\C}  \operatorname{Hom}_K(V_\tau, W) \longrightarrow W
\]
of the span of  finitely many summands $V_\tau \otimes _{\C}  \operatorname{Hom}_K(V_\tau, W)$ (here $V_\tau$ is the representation space for a representative of  $\tau \in \widehat K$).  We shall call the minimal set of $\tau{\in} \widehat K$ here the \emph{$K$-isotypical support} of $w {\in} W$.

   \begin{lemma} 
  \label{lem-K-finite-generating-family1}
  There exists a generating family of continuous sections for the continuous field $\{ C^*_r (G_t)\}$ consisting of smooth and compactly supported functions on $\mathbb{N}_GK$ that are $K$-finite for both the left and right translation actions of $K$ on $\mathbb{N}_GK$.
  \end{lemma}
  
  \begin{proof}
In the following argument we shall use  right and left convolutions of elements  $\phi\in C^\infty (K)$, or even $\phi\in L^2 (K)$, with elements  $f\in C_c^\infty (\mathbb{N}_KG)$, defined    by 
  \[
  (f*\phi)(g)=\int_K f(gk^{-1})\phi(k)\;dk
  \]
and 
   \[
  ( \phi*f)(g)=\int_K \phi(k)f(k^{-1}g)\;dk .
   \]
Both convolutions are smooth and compactly supported functions on $\mathbb{N}_GK$.  The formulas define $C^*$-algebra morphisms from $C^*_r(K)$ into the multiplier algebra of the $C^*$-algebra of continuous sections of $\{ C^*_r(G_t)\}$ that vanish at infinity.  

Given a smooth and compactly supported complex function $f$ on $\mathbb{N}_G K$ and $\varepsilon >0$, we can find smooth functions $\phi$ and $\psi$ on $K$ such that 
\[
\| f - \phi * f * \psi \| <\varepsilon ,
\]
where the norm is that of the $C^*$-algebra of continuous sections of $\{C^*_r (G_t)\}$ that vanish at infinity.  Now recall the Peter-Weyl isomorphism 
\[
C^*_r (K) \stackrel \cong \longrightarrow \bigoplus _{\tau \in \widehat K} \operatorname{End} (V_\tau).
\]
We can approximate $\phi$ and $\psi$, viewed as elements of $C^*_r (K)$, by elements $\phi_1 $ and $\psi_1$ of the $C^*$-algebra that map into the algebraic direct sum above, in such a way that 
\[
\| f - \phi _1* f * \psi_1 \| <\varepsilon  .
\]
The elements $\phi_1,\psi_1\in C^*_r (K)$ are automatically smooth functions on $K$, and the function $ \phi _1* f * \psi_1$ on $\mathbb{N}_G K$ is smooth, compactly supported and left and right $K$-finite.  The collection of all elements of this form, for all $f$ and all $\varepsilon>0$, is a generating family, as required.
\end{proof}

In the proofs of the following two lemmas we  shall use the fact that  if $f\in C_c^\infty (G)$ and $\phi\in C^\infty (K)$, then
\[
\pi_{\sigma,\nu}(f* \phi) = \pi_{\sigma,\nu}(f)\pi_{\sigma}(\phi) 
\quad \text{and} \quad 
\pi_{\sigma}(\phi* f) = \pi_{\sigma,\nu}(\phi) \pi_{\sigma,\nu}(f) ,
\]
where $\pi_{\sigma}(\phi) $ denotes the application to $\phi$ of  the representation of $C^*_r(K)$ associated to the restriction of the representation $\pi_{\sigma,\nu}$ to $K$; the latter is independent of $\nu$, and is simply the restriction of the left-regular representation of $K$ to $L^2 (K)^\sigma{\subseteq} L^2 (K)$. 

     \begin{lemma}
  \label{lem-K-finite-generating-family2}
Let $\{f_t\}$ be a  right $K$-finite   continuous section of  $\{ C^*_r (G_t)\}$. If for every $\sigma\in \widehat M_+$ the limit 
\[
\lim_{t\to 0} \pi_\sigma (\alpha_t (\lambda_t(f_t)))
\]
exists in $C_0  (\mathfrak{a}^*_{+,\sigma} , \Compacts (L^2 (K)^\sigma)  )$, then the limit 
\[
\lim_{t\to 0} \alpha_t (\lambda_t (f_t))
\]
exists in $C^*_r (G)$.
  \end{lemma}

  \begin{proof}
  By the Peter-Weyl theorem, the Hilbert space $L^2(K)^\sigma$ has a $K$-isotypical decomposition 
  \[
  L^2 (K)^\sigma \cong \bigoplus _{\tau \in \widehat K} V_\tau \otimes \operatorname{Hom}_K ( V_\tau,\C_\sigma)
  \]
It follows that an irreducible representation  $\tau {\in} \widehat K $, is included in the $K$-iso\-typical decomposition of $L^2 (K)^\sigma$ if and only if $\sigma$ is a weight of $\tau$, and therefore $\tau$ is included in only finitely many of the spaces $L^2 (K)^\sigma$, as $\sigma$ ranges over $\widehat M$.  So if $f$ is right $K$-finite, then the element 
\[
\pi_\sigma(\lambda_t(f_t)) \in C_0(\lie{a}^*_{+,\sigma}, \Compacts(L^2 (K)^\sigma )
\]
is nonzero for only a finite set of $\sigma{\in}\widehat M$  that is independent of $t$. Therefore under the hypotheses of the lemma the limit 
\[
\lim_{t\to 0} \bigoplus _{\sigma} \pi_\sigma ( \alpha_t (\lambda_t(f_t)))  =
 \bigoplus _{\sigma}  \lim_{t\to 0}  \pi_\sigma(\alpha_t(\lambda_t (f_t) )  )
\]
exists: we can commute the limit and the direct sum because only finitely many summands are nonzero.  The lemma follows from  the fact that $\oplus_\sigma \pi _\sigma$ is isometric.
  \end{proof}

\begin{lemma} 
  \label{lem-K-finite-generating-family3}
Let $\{f_t\}$ be a left and right $K$-finite   continuous section of  $\{ C^*_r (G_t)\}$ and let $\sigma\in \widehat M$. If  the limit 
\[
\lim_{t\to 0} \langle \phi, \pi_{\sigma , \nu} (\alpha_t (\lambda_t(f_t)))\psi \rangle
\]
exists for every $\phi,\psi\in L^2(K)^\sigma$, uniformly in $\nu\in \lie{a}^*$, then the limit 
\[
\lim_{t\to 0} \pi_{\sigma} (\alpha_t (\lambda_t(f_t)))
\]
exists in $C_0(\lie{a}^*_{+,\sigma}, \Compacts (L^2 (K)^\sigma)) $.
\end{lemma}
  
 \begin{proof} 
Let $S\subseteq \widehat K$ be the union of the $K$-isotypical supports of $\{f_t\}$ for the left and right translation actions, and let 
\[
L^2(K)^\sigma_S = \operatorname{Image} \Bigl ( \bigoplus_{\tau\in S}  V_\tau {\otimes} \operatorname{Hom}_K (V_\tau, L^2 (K)^\sigma)\longrightarrow  L^2 (K)^\sigma)\Bigr ) .
\]
This is a finite-dimensional subspace of $L^2 (K)^\sigma$, and the operators  $  \pi_{\sigma , \nu}(f_t) $ vanish on its orthogonal complement for all $\nu$ and all $t$.  If for any given $\nu\in \lie{a}^*$ the limits 
\begin{equation}
\label{eq-matrix-coefficient-limits}
\lim_{t\to 0} \langle \phi, \pi_{\sigma , \nu} (\alpha_t (\lambda_t(f_t)))\psi \rangle
\end{equation}
exist for all $\phi, \psi\in L^2(K)^\sigma_S$, then the limit 
\[
\lim_{t\to 0} \pi_{\sigma,\nu} (\alpha_t (\lambda_t(f_t)))
\]
exists in $\Compacts (L^2 (K)^\sigma)$.  If the limits \eqref{eq-matrix-coefficient-limits} exist uniformly in $\nu$ as $\phi$ and $\psi$ range over an orthonormal basis for $\phi, \psi\in L^2(K)^\sigma_S$, then the limit 
\[
\lim_{t\to 0} \pi_{\sigma} (\alpha_t (\lambda_t(f_t)))
\]
 exists in $C_0(\lie{a}^*_{+,\sigma}, \Compacts (L^2 (K)^\sigma))$, as required.
 \end{proof}
  
We shall use the following explicit formula for the Haar integral on $G_t$  (recall here  that  for $t\neq 0$ the Haar measure on $G_t$ is $|t|^{-d} dg$).
 
\begin{lemma}[See {\cite[Prop\,8.43]{Knapp02}}]
\label{lem-Haar-measure}
If $G=KAN$ is an Iwasawa decomposition, then the Haar measures on $K$, $\lie{a}$, and $\lie{n}$ can be normalized so that 
\[
\begin{aligned}
\quad \qquad \int _{G_t} f(g)\; dg
	& =\int_K \int_{\lie{a}}\int_{\lie{n}}  f(k\exp(X)\exp(Y)e^{2\rho(X)} |t|^{-d}dk\;dX\;dY 
	\\
	& =\int_K \int_{\lie{a}}\int_{\lie{n}}  f(k\exp(tX)\exp(tY)e^{2\rho(tX)} dk\;dX\;dY 	.\qquad \quad \qed
\end{aligned}
\]
\end{lemma}

 \begin{lemma}
\label{lem-matrix-coeff-formula}
Let $\sigma{\in} \widehat{M}_+$, let $\nu{\in}\lie{a}^*_{+,\sigma}$ and let $t{\ne} 0$.   If $f_t{\in} C^*_r (G_t)$ is represented by a smooth and compactly supported function on $G_t$, and if   $\phi,\psi{\in} C^\infty(K)^\sigma$, 
 then
\begin{multline*}
\bigl \langle\phi,  \pi_{\sigma,\nu}(\alpha_t (\lambda_t(f_t))) \psi\bigr \rangle
\\
= \int_{\lie{a}} \int_{\lie{n}}  
\bigl ( \phi^* {*} f_t {*} \psi \bigr ) \bigl ( \exp(-tY)\exp(-tX) \bigr ) 
    e^{ i\nu(X)}    e^{t\rho(X)}   \;  dX\;dY,
\end{multline*}
where $\phi^*(k) = \overline{\phi(k^{-1})}$.
\end{lemma}

\begin{proof}%[Proof of Lemma ~\ref{lem-matrix-coeff-formula}]
It follows from the  definitions of the scaling automorphism $\alpha_t$ and the morphism $\pi_\sigma$ that 
\begin{equation*}
  \bigl \langle \phi,  \pi_{\sigma,\nu}(\alpha_t (\lambda_t(f_t)))\psi \bigr \rangle  
  	= \int_{G} f_t(g)\bigl \langle\phi,  \pi_{\sigma,t^{-1}\nu}(g)\psi  \bigr \rangle  \; |t|^{-d}dg.
  \end{equation*}
If we  insert  into this formula the definition of the $L^2$-inner product, then we obtain
\begin{multline*}
 \int_{G} f_t(g)\langle\phi, \pi_{\sigma,t^{-1}\nu}(g)\psi\rangle  \;|t|^{-d}dg
 \\
 = \int_{G} f_t(g)
		\Bigl (  \int_{K} 
		  \phi^*(k^{-1}) \bigl ( \pi_{\sigma,t^{-1}\nu}(g)\psi\bigr )(k) 
		 \; dk \Bigr ) 
		\; |t|^{-d} dg  ,
\end{multline*}
and rearranging, and making the substitution  $g:=k\gamma^{-1}$ we get 
\begin{multline}
\label{eq-matrix-coefficient-calc1}
 \int_{G} f_t(g)\langle\phi, \pi_{\sigma,t^{-1}\nu}(g)\psi\rangle  \;|t|^{-d}dg
 \\
= \int_{K}  \int_{G}   \phi^*(k^{-1}) f_t(k\gamma^{-1})
		  \bigl( \pi_{\sigma,t^{-1}\nu}(k\gamma^{-1})\psi\bigr )(k) 
		 \; dk \; |t|^{-d} d\gamma   .
\end{multline}
      Now, according to the definition \eqref{eq-action-of-g-in-compact-form} of the principal series representations,
\[
 \bigl (\pi_{\sigma,t^{-1}\nu}(k\gamma^{-1})\psi \bigr)(k)
 =
 e^{-(\rho+it^{-1}\nu)H(\gamma)} \psi(\kappa(\gamma)) .
 \]
 Inserting this into the right-hand side of  \eqref{eq-matrix-coefficient-calc1}   we obtain
 \begin{equation}
 \label{eq-matrix-coefficient-calc2}
\int_K  \int_{G}  \phi^*(k^{-1})  f_t(k\gamma ^{-1})e^{(-\rho+it^{-1}\nu)H(\gamma )} \psi(\kappa(\gamma ))  
\; |t|^{-d} d\gamma \;dk .
\end{equation}
If we use the formula for the Haar measure on $G$ given in Lemma~\ref{lem-Haar-measure}, then we obtain from \eqref{eq-matrix-coefficient-calc2} the integral 
\begin{multline*}
  \int_K \int_{\lie{a}} \int_{\lie{n}} \int_K 
 \phi^* (k_1^{-1}) f_t\bigl (k_1\exp(-tY)\exp(-tX)k_2 ^{-1}\bigr ) \\
{\tiny\times} \,\, \phi(k_2 )     e^{(-\rho+it^{-1}\nu)(tX)}e^{2\rho(tX)} \;dk_1 \; dX\;dY\;dk_2 .
\end{multline*}
This is 
\begin{equation*}
  \int_{\lie{a}} \int_{\lie{n}}  
\bigl ( \phi^* {*} f_t {*} \psi \bigr ) \bigl ( \exp(-tY)\exp(-tX) \bigr ) 
    e^{i\nu(X)}    e^{\rho(tX)}   \;  dX\;dY  ,
\end{equation*}
as required.
\end{proof}

\begin{lemma} 
\label{lem-matrix-coeff-formula2}
Let $\sigma\in \widehat{M}_+$ and let $\nu\in \lie{a}^*_{+,\sigma}$. For any smooth and compactly supported function $f$ on $\mathbb{N}_GK$, and any $\phi, \psi \in L^2(K)^\sigma$, we have
\[
\lim_{t\rightarrow 0} \, \langle \phi,  \pi_{\sigma,\nu}(\alpha_t (\lambda_t(f_t)))\psi\rangle=\int_{\lie{a}}\int_{\lie{n}} ({\phi^*}{*}f_0{*}\psi)(e,X+Y)e^{i\nu(X)}\;dX\;dY.
\]
The convergence is uniform in $\nu\in  \mathfrak{a}^*_{+,\sigma}$.
\end{lemma}
 
  \begin{proof}
This follows immediately from Lemma~\ref{lem-matrix-coeff-formula} above and Lemma~\ref{lem-smooth-function-on-def-space}.
\end{proof}

\begin{proof}[Proof of Theorem~\ref{thm-the-limit-exists}]
According to Lemma~\ref{lem-reduce-to-generating-family}, we only need verify that the limit in the statement of the theorem exists for a generating family of continuous sections, and we shall use Lemma~\ref{lem-K-finite-generating-family1} to work with the generating family of continuous sections $\{ f_t\}$ associated to the   smooth, compactly supported, left and right $K$-finite functions on $\mathbb{N}_GK$.  Lemma~\ref{lem-matrix-coeff-formula2} shows that for every $\sigma{\in} \widehat{M}_+$ and every $\nu{\in} \lie{a}^*_{\sigma,+}$ the individual matrix coefficients of $\pi_{\sigma,\nu}(\alpha_t (\lambda_t (f)))$ converge to limits as $t{\to}0$, uniformly in $\nu$.  Lemmas~\ref{lem-K-finite-generating-family2} and \ref{lem-K-finite-generating-family3} complete the proof.\end{proof}

\subsection{Construction of an embedding morphism}
 
Let $f\in C^*_r(G_0)$. Extend $f$  in any way to a continuous section $\{f_t\}$ of $\{C^*_r(G_t)\}$ and then form the limit
\begin{equation}
\label{eq-embedding-formula}
\alpha (f)= \lim_{t\rightarrow 0} \alpha_t (\lambda_t(f_t)) 
\end{equation}
in $C^*_r(G)$.

\begin{theorem}
\label{thm-embedding}
The formula \eqref{eq-embedding-formula} defines an embedding of $C^*$-algebras 
\[
\alpha \colon C^*_r(G_0)\longrightarrow C^*_r (G).
\]
\end{theorem}

\begin{proof}
Since both $\lambda_t$ and $\alpha_t$ are isometric, 
\[
\| \lim_{t\rightarrow 0} \alpha_t (\lambda_t(f_t)) \| 
= 
\lim_{t\rightarrow 0} \| f_t \|  = \| f_0\| .
\]
Moreover if $\{ f'_t\}$ is a second extension of $f$ to a continuous section, then 
\[
\| \lim_{t\rightarrow 0}\alpha_t( \lambda_t (f_t)) -  \lim_{t\rightarrow 0}\alpha_t( \lambda_t (f'_t))\| 
=   \|  \lim_{t\rightarrow 0} \alpha_t( \lambda_t (f_t - f'_t)) \| 
= 0 .
\]
So the limit is independent of the extension, and it defines an isometric $*$-homomorphism, as required.
\end{proof}

\subsection{Mapping cone fields}  
\label{subsec-mapping-cone}
We begin with a very elementary construction: 
 
 \begin{definition}
 \label{def-mapping-cone-field}
 Let $\beta:B\rightarrow A$ be an embedding of a $C^*$-algebra $B$ into a $C^*$-algebra $A$. The \emph{mapping cone} continuous field of $C^*$-algebras over $\R$ associated to $\beta$ has fibers 
\[
\operatorname{Cone}(\beta)_{t}=\begin{cases}
A & t\neq 0\\
B & t=0.
\end{cases}
\]
Its continuous sections are all those set-theoretic sections $\{ f_t\}$ for which the function 
$$
t\mapsto  \begin{cases}
f_t & t\neq 0\\
\beta (f_0) & t=0 
\end{cases}
$$
from $\R$ to $A$ is norm-continuous. 
\end{definition}

We shall apply this construction to  the embedding  from Theorem~\ref{thm-embedding}.   

\begin{theorem} The fiber isomorphisms
\[
\begin{cases} 
\alpha_t\circ \lambda_t \colon C^*_r (G_t) \longrightarrow C^*_r (G) & t \ne 0 \\
\;\;\; \operatorname{id} \colon  C^*_r (G_0) \longrightarrow C^*_r (G_0) & t = 0
\end{cases}
\]
define an isomorphism of continuous fields
from the deformation field 
$\{C^*_r(G_t)\}$ to  the mapping cone field for the embedding 
\[
\alpha \colon C^*_r(G_0)\longrightarrow C^*_r (G).
\]
\end{theorem}
  
\begin{proof} It suffices to show that for any continuous section $\{f_t\}$ of $\{C^*_r(G_t)\}$, the image section of the mapping cone field is continuous; see  \cite[10.2.4]{Dixmier77}. But the  image section is $\{ \widehat f_t\}$, where 
\[
\widehat f_t =   
	\begin{cases} 
		\alpha_t(\lambda_t (f_t)) & t \ne 0 \\  
		f_0 & t=0 .
	\end{cases}
\]
This is obviously a continuous section of the mapping cone field away from $t{=}0$, and continuity at $t{=}0$ is proved using Theorem~\ref{thm-the-limit-exists} and the definition of $\alpha$.
\end{proof}

\section{Characterization of the Mackey bijection}
\label{sec-bijection-characterization}

The previous accounts of the Mackey bijection have all been organized around the concept  of \emph{minimal $K$-type} of an irreducible representation of $G$. Compare \cite{AfgoustidisAubert19}. In this final section we shall give a different treatment that is organized around the  embedding 
\[
\alpha \colon C^*_r (G_0) \longrightarrow C^*_r(G) ,
\]
and hence around the family of rescaling automorphisms $\{ \alpha _t \}$.

\subsection{The principal series as representations of the motion group}

\begin{lemma}
\label{lem-composition-of-pi-with-alpha}
The composition of the principal series representation $\Ind_P^G \sigma\otimes \nu$ of the connected complex reductive group $G$ with the morphism  
\[
\alpha \colon C^*_r(G_0) \longrightarrow C^*_r (G)
\]
is   the unitary representation $\Ind_{M\ltimes \lie{s}} ^{K\ltimes \lie{s}} \sigma\otimes \exp(i  \nu)$ of the motion group $G_0$.
\end{lemma}

\begin{proof}
The Hilbert space of the   representation $\pi=\Ind_{M\ltimes \lie{s}} ^{K\ltimes \lie{s}} \sigma\otimes \exp(i \nu) $ is the completion of the space of smooth functions $f\colon G_0 \to \C$ such that 
\[
f(g\cdot (m,X))= \sigma(m)^{-1} \exp(-i \nu (X)) f(g)
\]
in the norm associated to the inner product 
\[
\langle f_1,f_2 \rangle = \int _K \overline{ f_1(k)} f_2(k)\; dk .
\]
The action of $G_0$ is by left translation.  The Hilbert space identifies with $L^2 (K)^\sigma$ by restriction of functions to $K$, and in this realization the action of $G_0$ is 
\[
(\pi(k,X)\psi )(k_1) = \exp (i \nu(k_1^{-1}k\cdot X))\psi(k^{-1}k_1) .
\]
The matrix coefficient associated to $\phi,\psi \in L^2 (K)$ and $f\in C_c^\infty (G_0)$ is therefore 
\[
\langle \phi, \pi(f) \psi \rangle = \int _K \int_{K}\int_{\lie{s}} \overline{\phi(k_1)} f(k,X)  \exp (i \nu(k_1^{-1}k\cdot X))\phi(k^{-1}k_1) \;dk\;dk_1\; dX .
\]
Making the change of variables $k_2:=k_1^{-1}k$ we get 
\[
\langle \phi, \pi(f) \psi \rangle = \int _K \int_{\lie{s}} \int _{K}
\overline{\phi(k_1)} f(k_1k_2,X)  \exp (i \nu(k_2\cdot X))\phi (k_2^{-1}) \;dk_1 \; dX \: dk_2 ,
\]
and then the further change of variables $Z:=k_2{\cdot} X$ gives 
\[
\langle \phi, \pi(f) \psi \rangle = \int _K \int_{\lie{s}} \int _{K}
\overline{\phi(k_1)} f(k_1k_2,k_2^{-1}{\cdot} Z)  \exp (i \nu(  Z))\phi (k_2^{-1}) \;dk_1 \; dZ \: dk_2 .
\]
Now let us insert into the integral above the formula 
\[
\begin{aligned}
(\phi^*{*} f{*} \psi )(k,Z) 
	& =  \int _K \int _K \overline{\phi(k_1)} f(k_1 \cdot (k,Z)\cdot k_2)\psi(k_2^{-1}) \; dk_1 \; dk_2 \\
	& =  \int _K \int _K \overline{\phi(k_1)} f(k_1kk_2, k_2^{-1}{\cdot}Z)\ \psi(k_2^{-1}) \; dk_1 \; dk_2 .
\end{aligned}
\]
We obtain 
\[
\langle \phi, \pi(f) \psi \rangle =  
\int_{\lie{s}}  (\phi^*{*} f{*} \psi  )(e,Z) \exp (i \nu(  Z)) \; dZ.
\]
But Lemma~\ref{lem-matrix-coeff-formula2} shows that this is precisely $\langle \phi, \pi_{\sigma,\nu}{\circ} \alpha (f) \psi \rangle $,
and the proof is complete.
\end{proof}

In the following lemma and in the next subsection we shall make use of the classification of irreducible representations of the compact connected group $K_\nu$ by highest weights.  Rather than choose a dominant Weyl chamber for the action of the Weyl group $W_\nu$ on $\widehat M$, we shall   associate to a given irreducible representation $\tau$ the $W_\nu$ orbit of all possible highest weights for all possible choices of dominant Weyl chamber.  We shall use brackets, as in $[\theta]$, to denote this orbit.

The highest weight (or highest weight orbit) $[\theta]$ of an irreducible representation $\tau$ of $K_\nu$ is extremal in the following sense.  If $\sigma$ is any weight of $\tau$, then $\sigma$ lies in   the convex span of $W_\nu{\cdot} \tau$. Here, to form the convex hull, we view the free abelian group of all weights  as a lattice in a vector space via the embedding
\[
\widehat M \subseteq \widehat M\otimes_{\mathbb{Z}} \R.
\]
We shall write $[\sigma]\le [\theta]$ to denote the inclusion of $\sigma$ in the convex hull of $W_\nu {\cdot} \theta$. Of course,     this partial order depends on $\nu$, but the choice of $\nu$ will be clear from the context.

\begin{lemma}
\label{lem-decomp-into-irreps}
The composition of the principal series representation 
\[
\pi_{\sigma,\nu}=\Ind_P^G \sigma\otimes \exp(i \nu)
\]
 of the connected complex reductive group $G$ with the morphism  
\[
\alpha \colon C^*_r(G_0) \longrightarrow C^*_r (G)
\]
decomposes as a direct sum
\[
\bigoplus_{[\theta]\in \widehat M / W_\nu}
m(\sigma,\theta) \cdot  \pi (\theta, \nu) 
\]
as a representation of $G_0$, where the integer  $m(\sigma,\theta)$ is the multiplicity with which the weight $\sigma$ occurs in the representation of $K_\nu$ with highest weight $\theta$.
\end{lemma}

\begin{proof}
We showed in the Lemma~\ref{lem-composition-of-pi-with-alpha} that $\pi_{\sigma,\nu}{\circ}\alpha$ is the  induced representation $\Ind_{M\ltimes{\lie{s}}}^{K\ltimes \lie{s}} \sigma{\otimes} \exp(i \nu)$.  Let us analyze this representation by induction in stages \cite[Thm.\;3.3\;p.137]{Mackey55}: 
\[
\Ind_{M\ltimes{\lie{s}}}^{K\ltimes \lie{s}}\; \sigma{\otimes} \exp(i \nu) \cong 
\Ind_{K_\nu \ltimes{\lie{s}}}^{K \ltimes \lie{s}}\; \Ind_{M\ltimes{\lie{s}}}^{K_\nu \ltimes \lie{s}}\; \sigma{\otimes} \exp(i \nu) 
\]
As in the proof of Lemma~\ref{lem-composition-of-pi-with-alpha}, we can realize $ \Ind_{M\ltimes{\lie{s}}}^{K_\nu \ltimes \lie{s}} \sigma{\otimes} \exp(i \nu) $ on the Hilbert space $L^2 (K_\nu)^\sigma$, and in this realization an element $X\in \lie{s}$ acts as multiplication by the function 
\[
k\longmapsto \exp(i \nu (k^{-1} {\cdot} X))\qquad (\forall k\in K_\nu) .
\]
But if $k\in K_\nu$ then by definition, $\nu (k^{-1} {\cdot} X) = \nu(X)$.  So the subgroup $\lie{s}$ of $ K_\nu\ltimes \lie{s}$ acts on $L^2 (K_\nu)^\sigma$ by the unitary character $\exp(i\nu)$.  It follows that 
\[
 \Ind_{M\ltimes{\lie{s}}}^{K_\nu \ltimes \lie{s}} \sigma{\otimes} \exp(i \nu) = \bigoplus _{\tau\in \widehat K_\nu} m(\tau) \;\tau{\otimes}\exp(i \nu) ,
 \]
where $m(\tau)$ is the multiplicity with which $\tau{\in} \widehat K_\nu$ occurs in $L^2 (K_\nu)^\sigma$.  By the Peter-Weyl theorem (or Frobenius reciprocity)  $m(\tau)$ is also the multiplicity with which the weight $\sigma$ occurs in $\tau$; compare the proof of Lemma~\ref{lem-K-finite-generating-family2}.  The lemma follows from this and the classification of irreducible representations of the connected group $K_\nu$ by their highest weights.
\end{proof}

\subsection{Characterization of the Mackey bijection}

\begin{theorem}
There is a unique bijection 
\[
\mu \colon \widehat G_r  \longrightarrow \widehat G_0
\]
such that for every $\iota  \in \widehat G_r$, the element $\mu(\iota )\in \widehat G_0$ may be realized as a unitary subrepresentation of $\iota \circ \alpha$.
\end{theorem}

\begin{proof}
The existence part of the theorem is handled by the Mackey bijection from \cite{Higson08}, which is the map 
\[
\mu\colon \iota(\sigma, \nu)  \longmapsto \pi(\sigma, \nu) .
\]
Indeed by Lemma~\ref{lem-decomp-into-irreps}, the representation $\pi(\sigma, \nu)$ occurs within  $\iota( \sigma , \nu)$ with multiplicity one.

As for uniqueness, suppose we are given any bijection $\mu$, as in the statement of the theorem.  It follows from Lemma~\ref{lem-decomp-into-irreps} that $\mu$ must have the form 
\[
\mu  \colon \iota(\sigma, \nu )
\longmapsto   \pi(\theta, \nu)
\]
for some $\theta \in \widehat M / W _\nu$ with $\theta\ge \sigma$.  So for each fixed $\nu\in \mathfrak{a}^*$, we obtain from $\mu$ a bijection of sets 
\[
\mu _\nu \colon \widehat M \big / W_\nu 
\longrightarrow  \widehat M \big / W_\nu 
\]
defined by 
\[
\mu\colon \iota(\sigma, \nu) \longmapsto \pi(\mu_{\nu}(\sigma), \nu).
\]
We need to show that $\mu_\nu$ is the identity map for all $\nu$.

Now it follows from Lemma~\ref{lem-decomp-into-irreps} that $\mu^{-1}_\nu$  has the property that 
\[
\mu_\nu \bigl  ( [\sigma]\bigr ) \ge [\sigma] \qquad \forall\, [\sigma] \in \widehat M / W_\nu ,
\]
and so of course the inverse bijection 
has the property that 
\[
\mu_\nu^{-1}  \bigl  ( [\sigma]\bigr ) \le [\sigma]
\qquad \forall\, [\sigma] \in \widehat M / W_\nu .
\]
for all $\sigma  \in \widehat M / W_\nu $.   It follows from this that  $\mu^{-1}_\nu$ maps each of the finite sets 
\[
S_\sigma = \{\,  [\theta] \in \widehat M / W_\nu \, : \, [\theta] \le [\sigma] \, \}
\] into itself, and this in turn implies that 
$\mu^{-1}_\nu$ is the identity map, as required.
\end{proof}

\bibliography{References}
\bibliographystyle{alpha}

\end{document}